\documentclass[12pt, a4paper]{amsart}

\usepackage{accents,graphicx}

\newtheorem{theorem}{Theorem}

\theoremstyle{definition}

\theoremstyle{plain}
\newtheorem{thm}{Theorem}[section]
\newtheorem{lem}[thm]{Lemma}
\newtheorem{prop}[thm]{Proposition}
\newtheorem{corollary}[thm]{Corollary}

\title[Skein and asymptotic faithfulness]{Comparing skein and quantum group representations and their application to asymptotic faithfulness}

\author{Wade Bloomquist}
\email{bloomquist@math.ucsb.edu}
\address{Dept. of Mathematics\\
    University of California\\
    Santa Barbara, CA 93106-6105\\
    U.S.A.}

\author{Zhenghan Wang}
\email{zhenghwa@microsoft.com}
\address{Microsoft Station Q and Dept. of Mathematics\\
    University of California\\
    Santa Barbara, CA 93106-6105\\
    U.S.A.}

\thanks{The second author is partially supported by NSF grants DMS-1410144 and DMS-1411212.}

\begin{document}

\maketitle

\begin{abstract}

We make two related observations in this paper.  First the representations of mapping class groups from the Ising TQFT and its quantum group counterpart $SU(2)_2$ are neither equivalent as representations nor Galois conjugate to each other.  Hence mapping class
group representations obtained from quantum skein theory are fundamentally distinct
from those obtained from quantum group Reshetikhin-Turaev or geometric quantization
constructions.  Then we generalize the asymptotic faithfulness of the skein quantum $SU(2)$ representations of mapping class groups of orientable closed surfaces to skein quantum $SU(3)$.  We conjecture asymptotic faithfulness holds for skein quantum $G$ representations when $G$ is a simply-connected simple Lie group.  The difficulty for such a generalization lies in the lack of an explicit description of the fusion spaces with multiplicities to define an appropriate complexity of state vectors.

\end{abstract}

\section{Introduction}

Both classical and quantum topology in 3-dimensions are mature subjects, but deep connections between the two worlds are still rare.  Since the classical world is a limit of the quantum world as the Planck constant goes to zero, we expect classical topological information to emerge from quantum topology when the level $k$ of the physical quantum Witten-Chern-Simons theories (WCS) $(G,k)$ \cite{Witt} for a Lie group $G$ goes to infinity.  The asymptotic faithfulness of mapping class group representations is such a phenomenon, and so would be the volume conjecture.  Representations of mapping class groups can be constructed either using the skein method or the representation categories of quantum groups at roots of unity (see Section 6 of \cite{RW} and the references therein).  In this paper, we make two related observations.  First the representations of mapping class groups from the Ising TQFT and its quantum group counterpart $SU(2)_2$ are neither equivalent as representations nor Galois conjugate to each other.  Hence skein quantum representations of mapping class groups are different from the Reshetikhin-Turaev (RT) ones from quantum groups or geometric quantization.

The asymptotic faithfulness of quantum skein  $SU(2)$ is proved in \cite{FWW},  which is independent of, and different from, the  parallel approach using quantum $SU(n)$ RT- TQFTs\footnote{The account of the proof of the asymptotic faithfulness in Section 5 of \cite{MBomb} does not match the recollections of the second author.} \cite{And}.  The mathematical realizations of quantum WCS theories $(G,k)$ are commonly considered to be the RT-TQFTs from the modular tensor categories (MTCs) $G_k$, which are constructed using quantum groups at roots of unity \cite{RT1,RT2,Tur,RSW}.  A closely related construction of TQFTs is through skein theory related to $SU(2)$, called the Jones-Kauffman (JK) TQFTs \cite{Tur, BHMV, Wang}, where the Temperley-Lieb algebras as rediscovered by Jones \cite{Jon} and reformulated by Kauffman \cite{KL} using diagrams are exemplars of skein theories.
Contrary to some claims in the literature, these two versions of TQFTs, though closely related, are different as their underlying MTCs are different.  For example, in the $SU(2)$ level=2 case, the MTC $SU(2)_2$ for RT-TQFT has central charge $c=3/2$, while the skein JK-MTC is the Ising category with  $c=1/2$ (see the explicit data 5.3.4 and 5.3.5 on page 376 of \cite{RSW}), though both theories share the same modular $S$ matrix.  More concretely, in the mapping class group representation of the torus the image of the Dehn twist along the meridional curve is the $T$-matrix of the MTCs, so we have
\[T_{Ising}=\left(\begin{array}{ccc}
1 & 0 & 0\\
0 & e^{\frac{2\pi i}{16}} & 0 \\
0 & 0 & -1\\
\end{array}\right),\]
 and
\[T_{SU(2)_2}=\left(\begin{array}{ccc}
1 & 0 & 0\\
0 & e^{\frac{6\pi i}{16}} & 0 \\
0 & 0 & -1\\
\end{array}\right),\] respectively.   In general, there are non-trivial Frobenius-Schur indicators in the RT-TQFTs, which make them hard to work with using graphical calculus.  We choose to work with the JK-TQFTs and their skein generalizations due to their elementary combinatorial nature, and study the asymptotic faithfulness of their representations of the mapping class groups of orientable closed surfaces.

A salient feature of $SU(2)$-TQFTs is the multiplicity-freeness of all fusion spaces: the dimension of the space for a labeled pair of pants is either $0$ or $1$.  A convenience for skein quantum $SU(2)$  is that all simple objects are self-dual with trivial Frobenius-Schur indicators.  It follows that the skein theory requires only unoriented simple closed curves and unoriented trivalent graphs with uncolored trivalent vertices.   Then the complexity of a state vector for an induction has obvious topological meaning as the measure of the number of parallel simple curves.  To generalize to $SU(3)$, we have to use oriented curves and colored trivalent vertices.  Once we identify an appropriate complexity for state vectors in skein quantum $SU(3)$, our proof is completely analogous to the $SU(2)$ case provided we have the explicit knowledge of the fusion spaces given by the generalizations of Jones-Wenzl projectors.

The difficulty in extending this result to the other rank $2$ spiders of Kuperberg \cite{Kup} lies in obtaining an explicit description of the fusion spaces.  The construction of a basis, a key result obtained in \cite{Suc}, allows our argument to be carried through.  In particular, for $SU(3)$ the parameter $\ell$ described below is a crucial simplification that is not easily extended to the other rank $2$ spiders.

In section 2, we show that the two representations of the mapping class groups of oriented closed surfaces from RT-TQFT \cite{RT1,RT2} and JK-TQFT \cite{KL, Tur, BHMV} at level=$2$ are neither Galois conjugates of each other nor equivalent to each other as representations.    Then we generalize the faithfulness from skein quantum $SU(2)$ to skein quantum $SU(3)$ following the same strategy of \cite{FWW} in Section 4.

We recall a question posed by Turaev \cite{Tur}:  \lq\lq Is the (projective) action of the mapping class group of a closed oriented surface $\Sigma$ in the module of states of $\Sigma$ irreducible?  Consider the kernels of these actions corresponding to all modular categories.  Is the intersection of these kernels non-trivial? (This would be hard to believe.)"  In the specific case of skein $G_k$ TQFTs for simply-connected Lie group $G$, we conjecture that asymptotic faithfulness holds, meaning that the intersection of the kernels are the central elements of the mapping class group.

\section{Skein vs Quantum Group}

Two important methods to construct MTCs are the quantum groups at roots of unity and skein theoretical method (see Section 6 of \cite{RW} and the references therein).  Associated to $SU(2)$, they are the MTCs corresponding to the RT- and JK- TQFTs, respectively.  A common misconception is that they are Galois conjugates of each other.  These will not be the case as in general Galois conjugation may send a unitary theory to a non-unitary one, but both the RT- and JK- MTCs are unitary ones.  In this section, we will explicitly demonstrate this using the Ising and $SU(2)_2$.  Then we will show that the representations of mapping class groups are also not equivalent as representations.  Therefore, the Ising and $SU(2)_2$ representations are neither equivalent as representations nor Galois conjugate to each other.

\subsection{Galois twist}

As shown in \cite{DHW}, every MTC $\mathcal{B}$ has a defining number field $K_{\mathcal{B}}$, though not necessarily canonical.  Then every Galois automorphism $s$ of $K_{\mathcal{B}}$ leads to another MTC $\mathcal{B}_s$ by applying $s$ to all the defining data of $\mathcal{B}$.

Using the data from \cite{RSW} (see the explicit data 5.3.4 and 5.3.5 on page 376 of \cite{RSW}), we can see that both Ising and $SU(2)_2$ can be defined using the number field $\mathbb{Q}(q),q=e^{\frac{2\pi i}{16}}$.  To change Ising to $SU(2)_2$, we have to send $q$ to $s(q)=q^3$ in order to match the $T$-matrix.  But then $s(\sqrt{2})=-\sqrt{2}$, hence $s$ sends the $s$-matrix of Ising to a theory with a negative quantum dimension, not $SU(2)_2$.  This is a general phenomenon for Galois conjugate.  Therefore, Ising and $SU(2)_2$ are not Galois conjugates of each other\footnote{As a referee pointed out that another argument would be to compute the second FS-indicators
for the $\sigma$ particle in $SU(2)_2$ and Ising because no Galois automorphism can take 1 to -�1.}.

\subsection{Ising vs $SU(2)_2$ representations}

In this subsection, more explicit data that distinguish the skein quantum $SU(2)$  at level $k$ from RT-$SU(2)_k$ representations are provided.  We focus on the Ising and $SU(2)_2$ modular tensor categories, which give rise to inequivalent representations of the mapping class group representations, and consequently different $3$-manifold invariants of mapping tori.   Ising and $SU(2)_2$ have the same fusion rules with the same label set $L=\{0,1,2\}$ and quantum dimensions $d_0=d_2=1$ and $d_1=\sqrt{2}$.  The nontrivial fusion rules are
\[1\otimes 1=0+2, 1\otimes 2=2\otimes 1=1, 2\otimes 2=0.\]   They are different modular tensor categories with the same $S$ matrix
\[S=\left(\begin{array}{ccc}
1 & \sqrt{2} & 1\\
\sqrt{2}  & 0 & -\sqrt{2}\\
1 & -\sqrt{2} & 1\\
\end{array}\right),\] but  different topological twists.  In the Ising category $\theta^{Ising}_0=1, \theta^{Ising}_1=e^{\pi i/8}, \theta^{Ising}_2=-1$, but in $SU(2)_2$, $\theta^{SU(2)_2}_0=1, \theta^{SU(2)_2}_1=e^{3\pi i/8}, \theta^{SU(2)_2}_2=-1$.

First we consider the genus=$1$ torus case.  The mapping class group of the torus is $SL(2,\mathbb{Z})$ with standard generators $s,t$.  The assignment of $\frac{S}{2}$ to $s$ and $T$ to $t$ leads to a projective representation of $SL(2,\mathbb{Z})$ whose projective phases are powers of  $e^{\frac{\pi i c}{4}}$, where $c$ is $\frac{1}{2}$ for Ising and $\frac{3}{2}$ for $SU(2)_2$ \cite{Tur, RSW}.

\begin{theorem}
Suppose $x \in SL(2,\mathbb{Z})$ is sent to a matrix $M_{\alpha}(x)$ of the form $M_{\alpha}(x)=S^{a_1}T_\alpha^{b_1}\cdots S^{a_n}T_\alpha^{b_n}$ such that $\sum_{j=1}^n a_j$ is even, where $\alpha$ is either Ising or $SU(2)_2$.  If $\sum_{j=1}^n b_j=2$ mod $4$, then
\[|Tr(M_{Ising})|\neq |Tr(M_{SU(2)_2})|. \]
\end{theorem}

\begin{proof}
First we prove that for any $m=2$ mod $4$,
\[|Tr(T^m_{Ising})|\neq |Tr(T^m_{SU(2)_2})|. \]
For the Ising theory, setting $m=4a+2$,
we have that
\[T_{Ising}^m=\left(\begin{array}{ccc}
1 & 0 & 0\\
0 & e^{\pi i(2a+1)/4} & 0\\
0 & 0 & 1\\
\end{array}\right)\]
so
\[Tr(T_{Ising}^m)=2+e^{\pi i(2a+1)/4}.\]
It follows that
\[|Tr(T_{Ising}^m)|=\sqrt{(2+\cos(\frac{\pi(2a+1)}{4}))^2+\frac{1}{2}}.\]
When $2a+1= 1 \mod 8$ or $2a+1= 7 \mod 8$,
\[|Tr(T_{Ising}^m)|=\sqrt{5+2\sqrt{2}}.\]
When $2a+1= 3\mod 8$ or $2a+1= 5\mod 8$,
\[|Tr(T_{Ising}^m)|=\sqrt{5-2\sqrt{2}}.\]

For the $SU(2)_2$ theory,
\[T_{SU(2)_2}^m=\left(\begin{array}{ccc}
1 & 0 & 0\\
0 & e^{3 \pi i(2a+1)/4} & 0\\
0 & 0 & 1\\
\end{array}\right)\]
so
\[Tr(T_{SU(2)_2}^m)=2+e^{3 \pi i(2a+1)/4},\]
so we have
\[|Tr(T_{SU(2)_2}^m)|=\sqrt{(2+\cos(\frac{3\pi(2a+1)}{4})^2+\frac{1}{2}}.\]
When $2a+1= 1 \mod 8$ or $2a+1= 7 \mod 8$,
\[|Tr(T^m)|=\sqrt{5-2\sqrt{2}}.\]
When $2a+1= 3\mod 8$ or $2a+1= 5\mod 8$,
\[|Tr(T^m)|=\sqrt{5+2\sqrt{2}}.\]

For the general case $M$, by
using $Tr(AB)=Tr(BA)$ and  $S^2$ being sent to a phase factor of norm=$1$, we see that all $S$ factors can be dropped as $\sum_{j=1}^n a_j$ is even.  Therefore, the general case reduces to the case that $T$ appears $2 \mod 4$ times.

\end{proof}

Note that the theorem implies that the three manifold invariants from $SU(2)_2$ and Ising are different for certain lens spaces.

\par
The genus $1$ case implies similar results for higher genera.  For example, the Reshetikhin-Turaev vector space associated to a genus two surface is the space of admissible labellings of the spine of the handlebody bounded by the genus $2$ surface.  We will denote a basis element by $(a,b,c)$ as seen in Figure \ref{g2basis}.

\begin{figure}
\includegraphics[scale=.4]{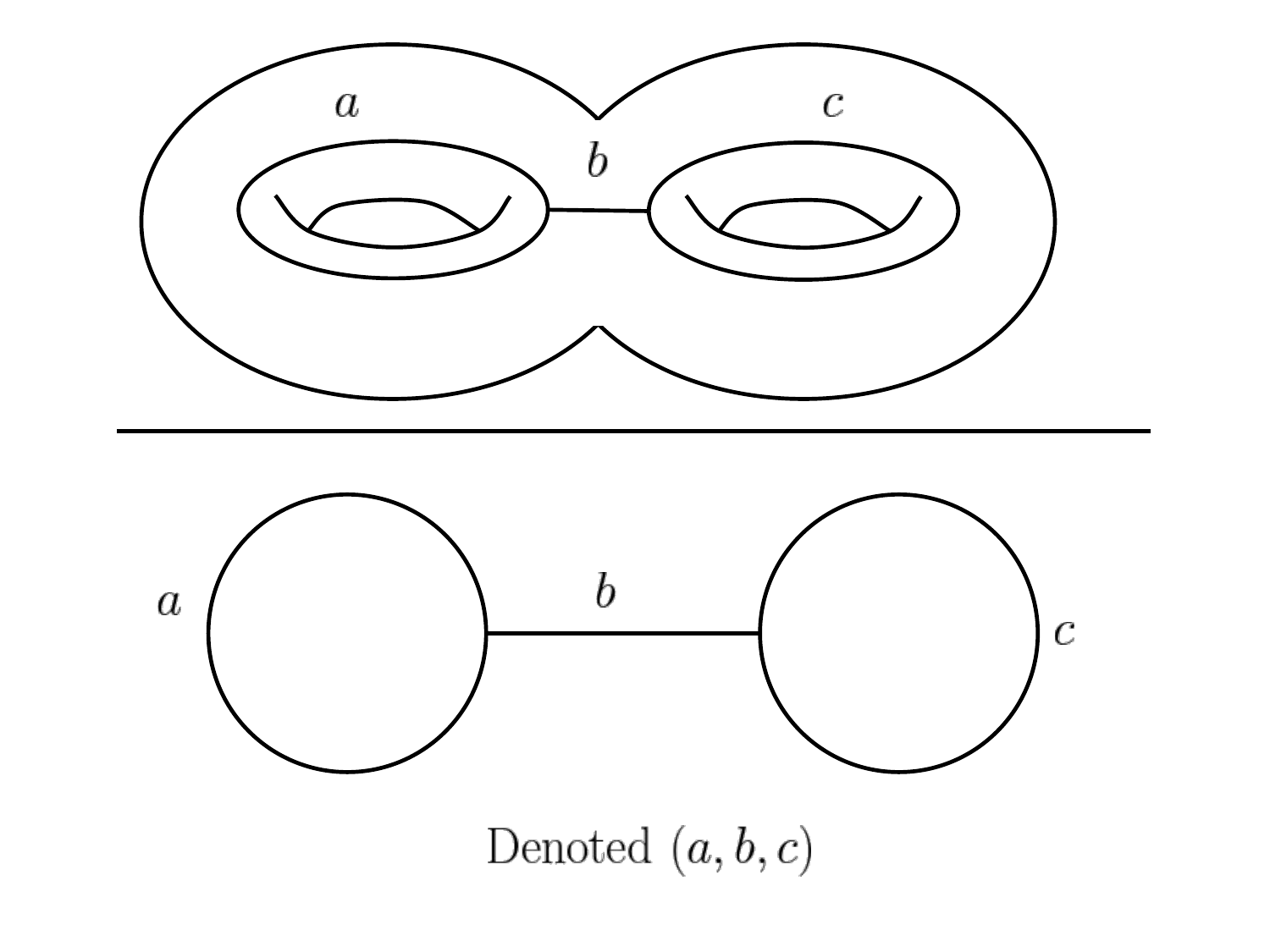}
\caption{•}\label{g2basis}
\end{figure}

To work explicitly with matrices, we write down an ordered basis using the fusion rules shared by Ising and $SU(2)_2$:
\[v_1=(0,0,0), v_2=(1,0,0), v_3=(2,0,0)\]
\[v_4=(0,0,1), v_5=(1,0,1), v_6=(2,0,1)\]
\[v_7=(0,0,2), v_8=(1,0,2), v_9=(2,0,2)\]
\[v_{10}=(1,2,1).\]
Our mapping classes of interest will be analogues of the $T$ matrix in this higher genus setting.  Take $M$ to be the Ising representation of the positive Dehn twist along the meridional curve, $\gamma_1$, that bounds the disk which is dual to the edge of the spine labeled $a$, as seen in Figure \ref{meridian}.  Similarly define $M^\prime$ to be the $SU(2)_2$ analogue.

\begin{figure}
\includegraphics[scale=.4]{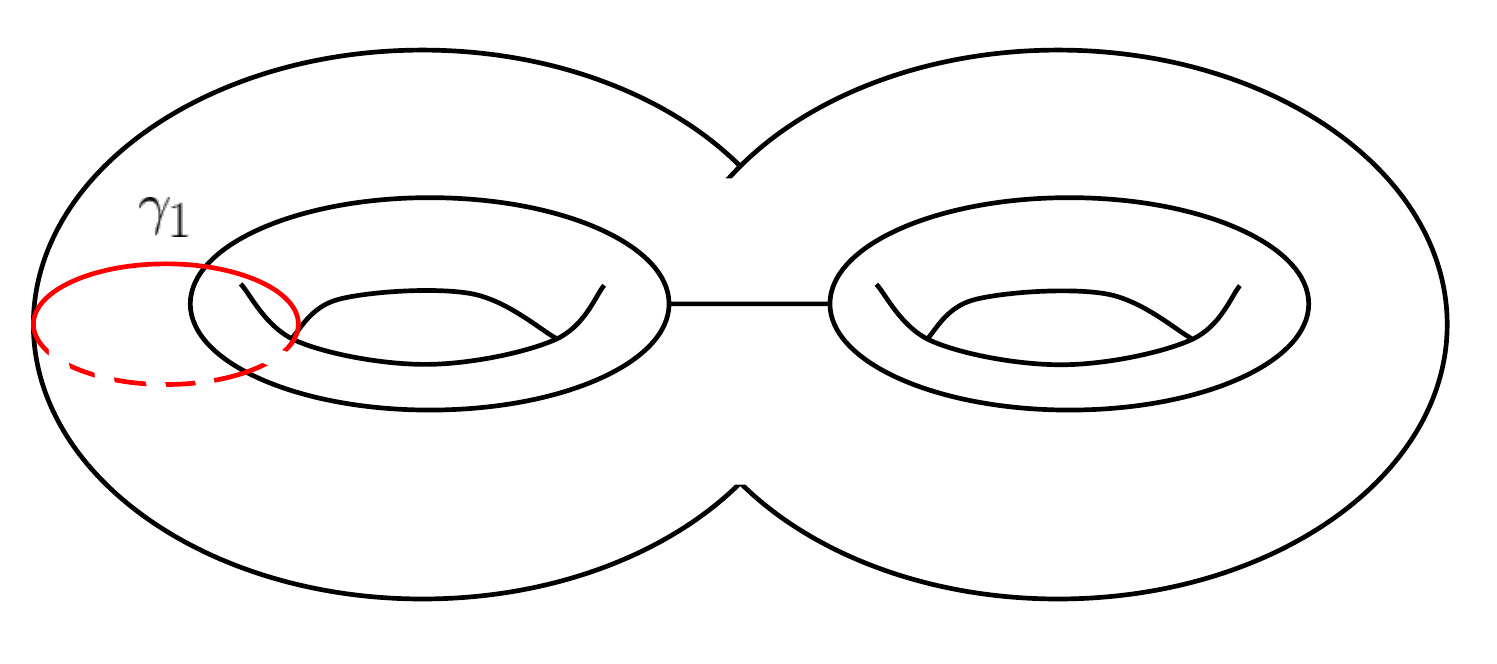}
\caption{•}\label{meridian}
\end{figure}

Explicitly, the action on this basis is given by $M(a,b,c)=\theta_a (a,b,c)$.  Thus,
\[M=diag[1,e^{\pi i/8},-1,1,e^{\pi i/8}, -1, 1, e^{\pi i/8}, -1, e^{\pi i/8}]\]
and
\[M^\prime=diag[1,e^{3\pi i/8},-1,1,e^{3\pi i/8}, -1, 1, e^{3\pi i/8}, -1, e^{3\pi i/8}].\]
\begin{corollary}
With $M$ and $M^\prime$ as defined above,
\[|Tr(M^m)|\neq |Tr((M^\prime)^m)|\]
for $m= 2\mod 4$.
\end{corollary}

\section{Skein $SU(3)$ TQFTs}

\subsection{$SU(3)$ Skein Theory}

\subsubsection{Kuperberg's spiders}
The spiders are skein constructions of the representation theory of a simple Lie algebra (and their associated quantum groups).  These have explicit combinatorial descriptions when the Lie algebra is of rank $2$, which provide a generalization of the Temperley-Lieb algebra for $A_1$ to $A_2, B_2,$ and $G_2$.  We will only focus on the case of $A_2$, as our interest will be in $U_q(\mathfrak{sl}_3)$.  The $A_2$ spider is generated by the following two webs seen in Figure \ref{webg}

\begin{figure}[ht!b]
\centering
\includegraphics[scale=.2]{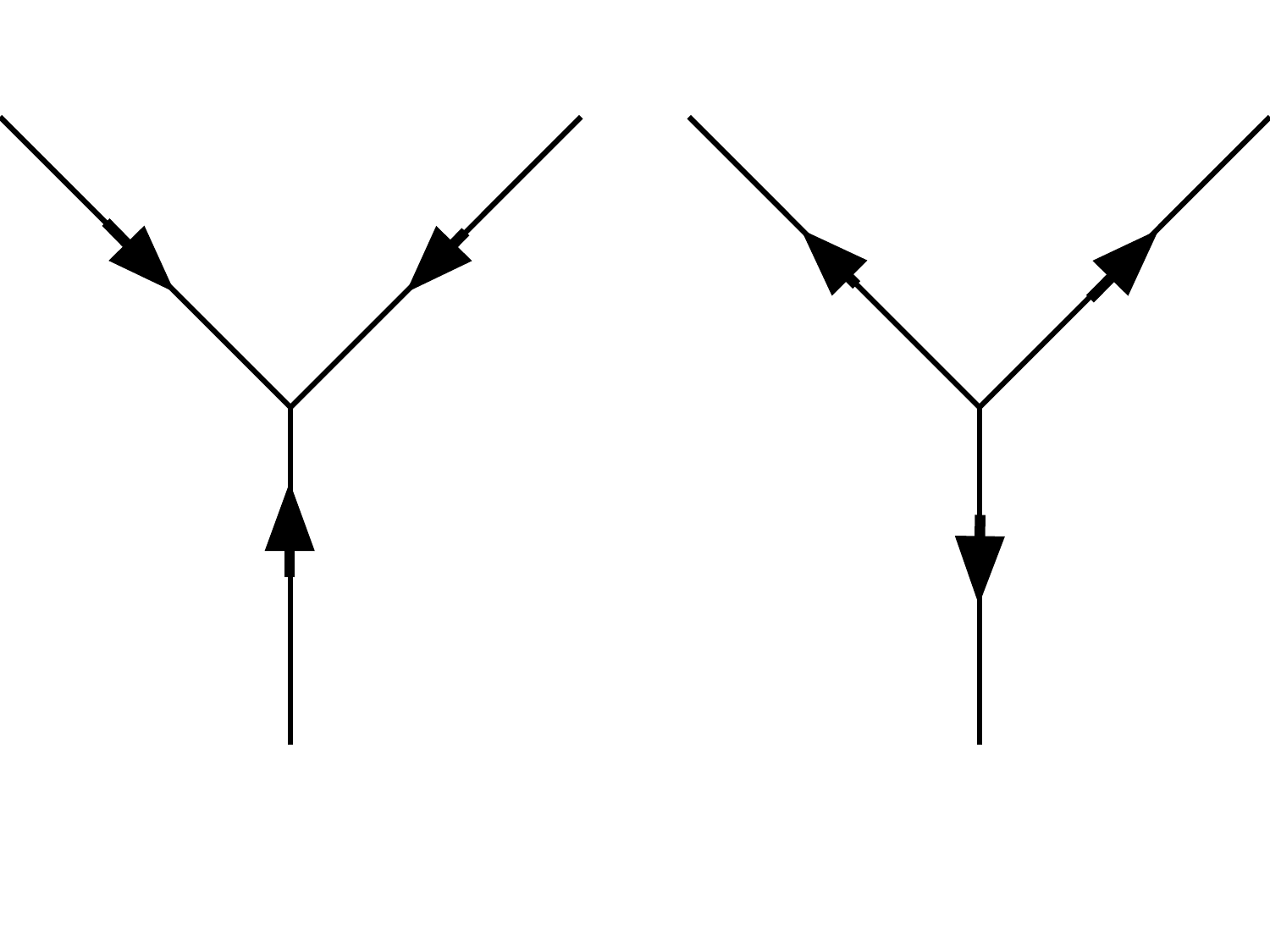}
\caption{Sinks and Sources}\label{webg}
\end{figure}

subject to the following relations in Figure \ref{webr}.

\begin{figure}[ht!b]
\centering
\includegraphics[scale=.3]{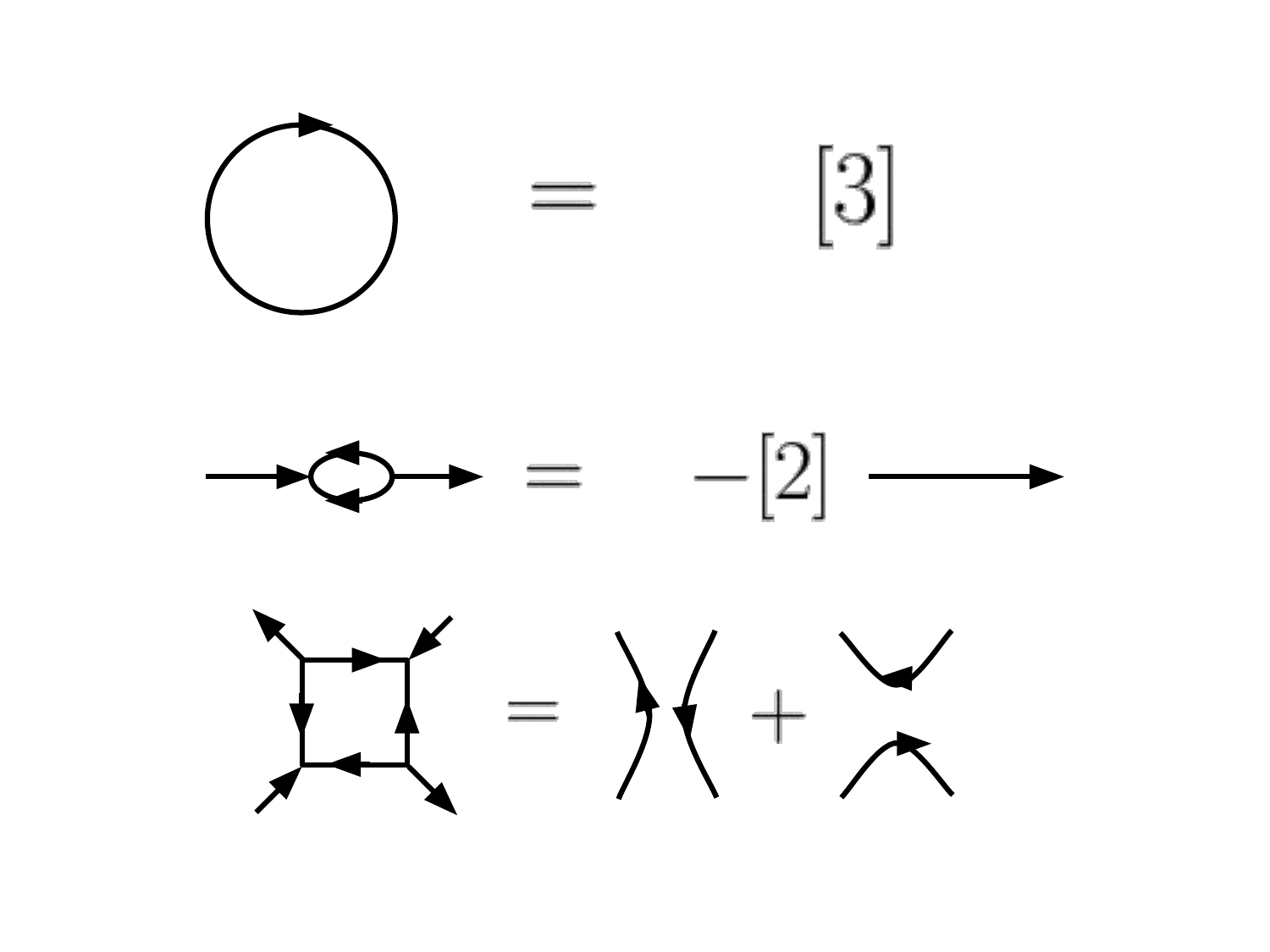}
\caption{$SU(3)$ Relations}\label{webr}
\end{figure}

Kuperberg abstractly proved the existence of generalized \lq\lq Jones-Wenzl idempotents" for the simple Lie algebras of rank=$2$, which are called clasps \cite{Kup}.

\subsubsection{Generalization of Jones-Wenzl Idempotents}
The work of  Kim developed the theory of Jones-Wenzl idempotents for rank $2$ Lie algebras, and in particular for $A_2$.  A graphical recursion was found for these projectors \cite{Kim}.
These are the minimal projectors in the skein algebra of the disk.  They satisfy the annihilation axiom as well, under both $Y's$ and caps.  We introduce the notation $(m,n)$ for the fundamental projector indexed by $m$ and $n$.
\begin{theorem}
For $a,b\geq 1$, the fundamental projectors satisfy the recursion given in Figure $5$.
\end{theorem}

\begin{figure}[ht]
\centering
\includegraphics[scale=.45]{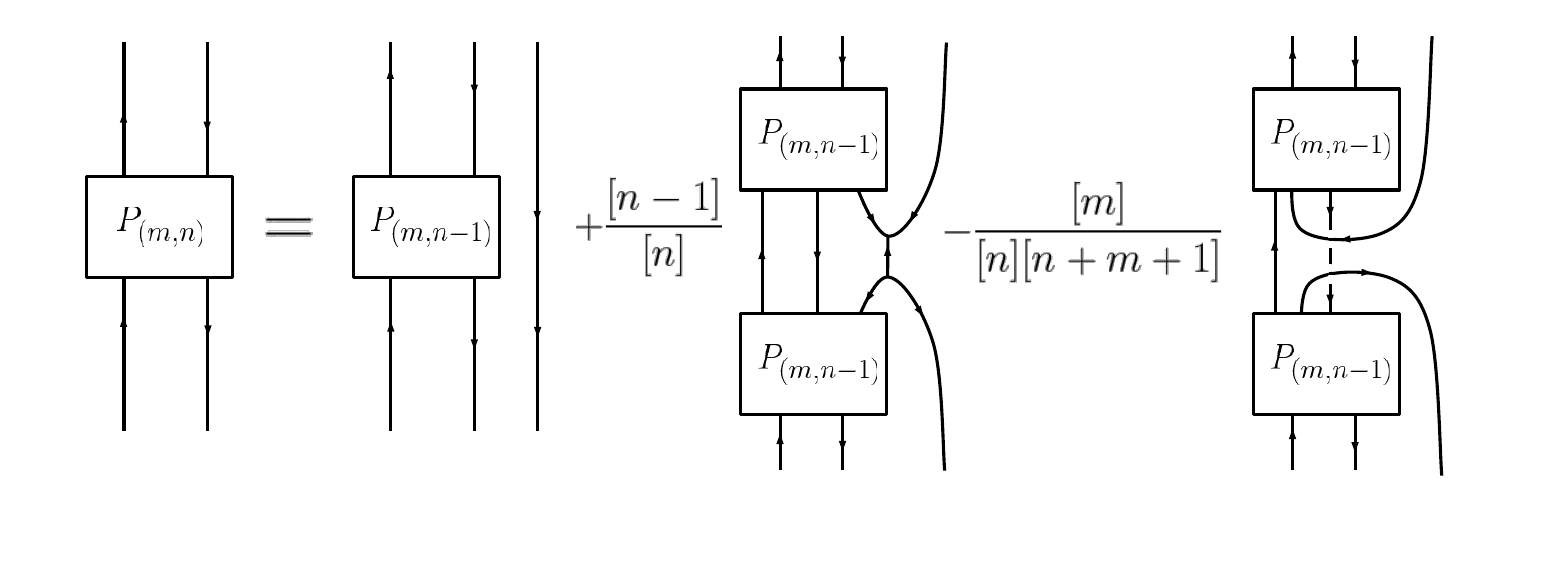}
\caption{The fundamental projectors for $A_2$}
\end{figure}
Where we note that the crossing is used for convenience and expanding into the standard basis would and using the annihilation property of the projectors would yield the unique maximal cut out from the hexagonal tiling with the appropriate boundary.
\subsubsection{Skein $SU(3)_k$ Intertwiner or Fusion or Triangle Spaces}
To build up the theory of the intertwiner spaces in the $A_2$ web in the classic case, we turn our attention to
\[I((m_1,n_1),(m_2,n_2),(m_3,n_3))\]
\[\cong Hom((m_1,n_1)\otimes(m_2,n_2)\otimes (m_3,n_3),\mathbb{C})\]
\[\cong Hom((m_1,n_1)\otimes (m_2,n_2),(n_3,m_3)).\]

Note that for $SU(3)$,  the dual of $(m,n)$ is $(n,m)$ and the above equation defines $I((m_1,n_1),(m_2,n_2),(m_3,n_3))$.  Corresponding to the disc with $m_1+m_2+m_3$ entries and $n_1+n_2+n_3$ exits, where these boundary points are organized into three groups, the projector $(m_i,n_i)$ is placed on the corresponding grouping.  Up to an invertible scalar from above, this construction gives us an isomorphism class of vector spaces, called the intertwiner space of type $((m_1,n_1),(m_2,n_2),(m_3,n_3))$, denoted $I((m_1,n_1),(m_2,n_2),(m_3,n_3))$, which will be referred to as the fusion or triangle space $((m_i,n_i))$.
\par

A necessary condition for the intertwiner spaces to be nontrivial is
\[|(m_1+m_2+m_3)-(n_1+n_2+n_3)|=3\ell\]
where $\ell \geq 0$ is a parameter of the space.  Suppose that $m_1+m_2+n_2=n_1+n_2+m_3=s$, and let $\ell\geq 0$ be as above, then denote by $I_\ell(m_i,n_i)$ the triangle space $I((m_i+\ell,n_i))$, and by $I^\ell(m_i,n_i)$ the triangle space $I((m_i,n_i+\ell))$.  Define $p_i=s-m_i-n_i$.

\begin{thm}
\begin{enumerate}
\item $I_\ell((m_i,n_i))$ is nontrivial if and only if $p_i\geq 0$
\item $dim(I_\ell((m_i,n_i)))=\min(m_i,n_i,p_i)+1.$
\end{enumerate}
\end{thm}
A proof of the theorem is in \cite{Suc}.  The proof revolves around computing theta symbols through a fairly involved recursion.  This is analogous to computing the theta symbols in Temperley-Lieb recoupling theory.

An admissible $6-$tuple $(x,y,a,b,u,v)$ is an ordered set of six non-negative integers such that the following holds
\[a+v=m_1,\quad x+u=n_1\]
\[u+b=m_2, \quad v+y=n_2\]
\[a+b=m_3, \quad x+y=n_3\]
There are exactly $\min(m_i,n_i,p_i)+1$ of these admissible $6-$tuples, which form a basis for $I((m_i,n_i))$.
This particular representation of the intertwiner space has been chosen as it will be most hopeful when applying to the construction described below.  A second, potentially more precise figure, utilizes the diagrammatic trick of \lq\lq pillows".  A pillow allows for the reordering of boundary strands, and can be thought of as an isomorphism of the skein space of the disk depending on the realization we are using.

\begin{figure}[ht]
\centering
\includegraphics[scale=.35]{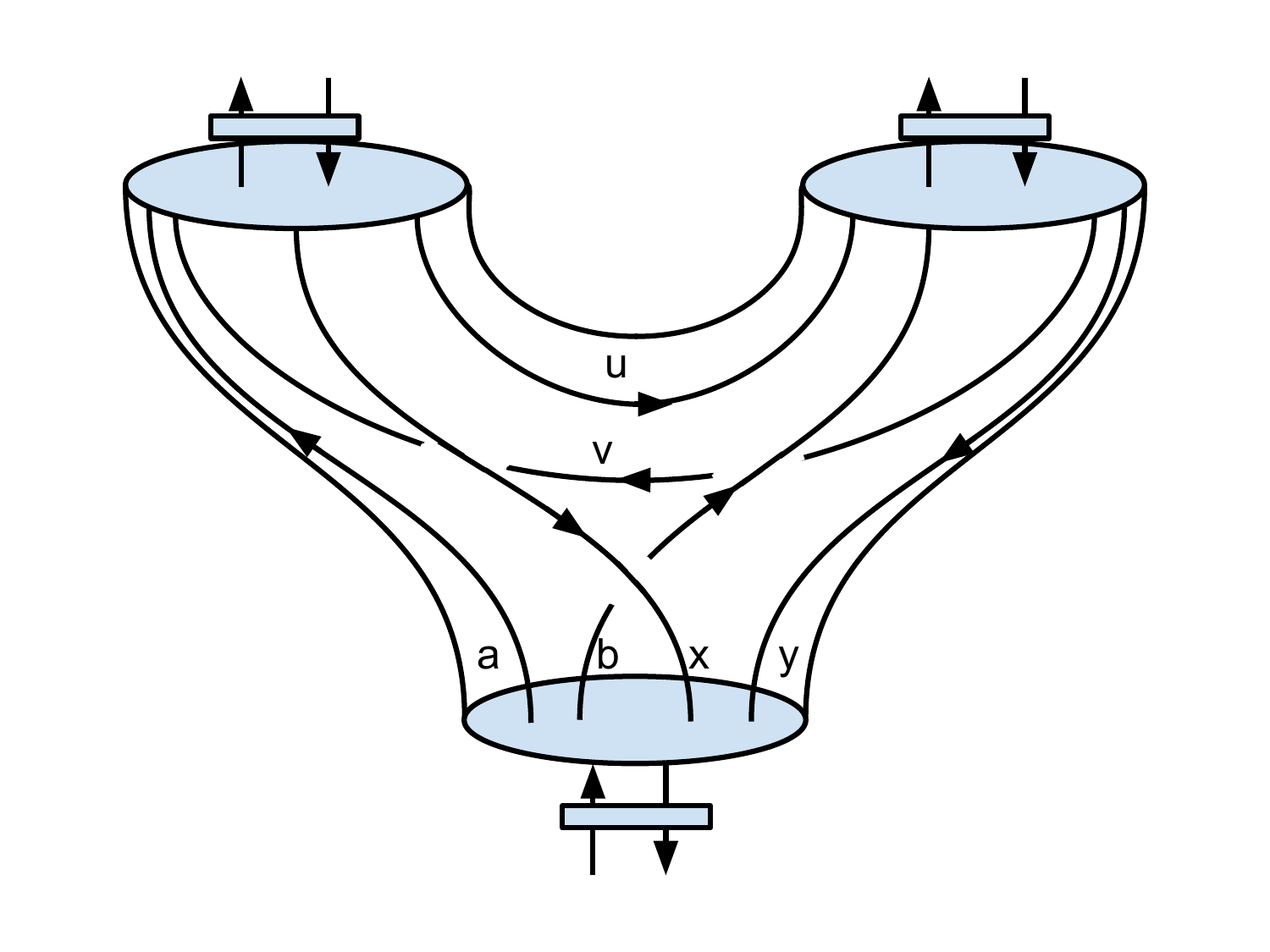}
\caption{$I((m_i,n_i))$}
\end{figure}

\begin{figure}[ht]
\centering
\includegraphics[scale=.35]{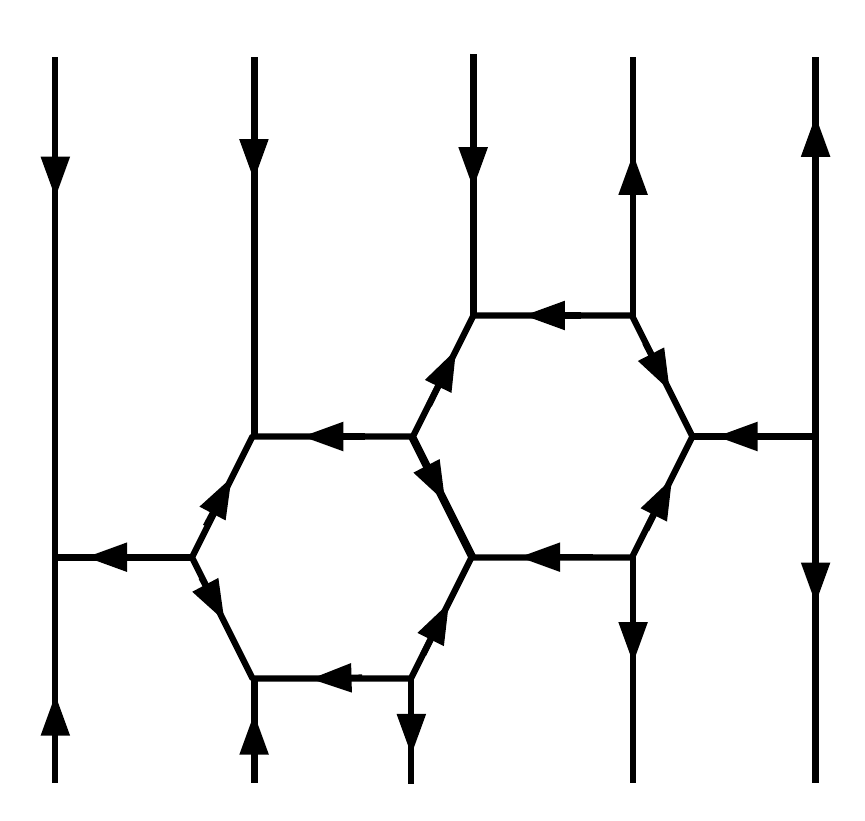}
\caption{A pillow changing the ordering of boundary points from $(2,3)$ to $(3,2)$}
\end{figure}

\begin{figure}[ht]
\centering
\includegraphics[scale=.35]{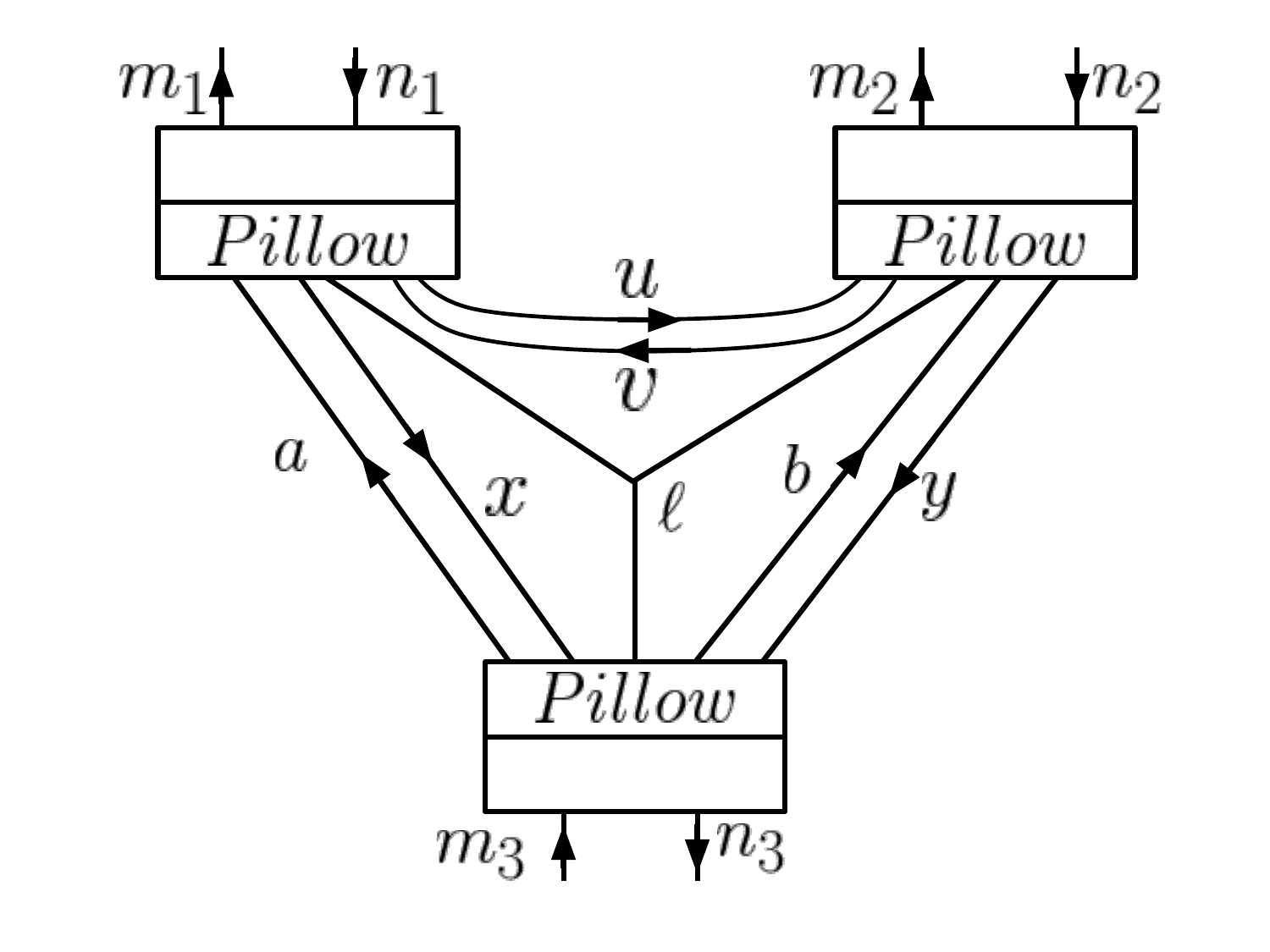}
\caption{$I(m_i,n_i)$ using a pillow construction}
\end{figure}

The case of $\ell\neq 0$ can be seen by adding in an $\ell-$fold triple point in the skein space as well.These two realizations of the same underlying combinatorial basis correspond to the ability of a braiding to switch strands, as annihilating \lq\lq turn backs" will appear in all but one of the terms when expanding the crossing in terms of basis elements.

 \subsubsection{Numerical data}
 Fix a level $k$, and let $A$ be a $6r^{\text{th}}$ primitive root of unity where $r=k+3$.  In particular,
\[[n]=\frac{A^{3n}-A^{-3n}}{A^3-A^{-3}}.\]
This leads to new identities in the level $k$ instance
\[[3r]=0,\qquad [3r-n]=[n],\qquad [3r+n]=-[n], \qquad [n+6r]=[n].\]
Recall that
\[tr((m,n))=\langle (m,n),(m,n)\rangle=\frac{[m+1][n+1][m+n+2]}{[2]}\]

Now we  impose the condition that if
\[\langle \xi,\eta\rangle=0\]
 for all $\eta$, then the relation $\xi=0$ is added to the theory, where $\eta,\xi$ are elements of the same disc space.
This implies that as long as $m+n\leq k$, the fundamental projectors are identical at level $k$ to the classical case, which immediately implies that as long as $s+\ell\leq k$ the triangle spaces are similarly the same as the classical case. Proofs are contained in chapter $5$ of  \cite{Suc}.

\subsection{Construction of TQFT State Spaces $V(Y)$}
TQFTs can be constructed from a manifold invariant generated by a skein theory satisfying some properties, which can be found in \cite{Ati, BHMV}.  Lickorish established that the $SU(N)$ skein theory of Yokota satisfies these properties and can be used to build such a TQFT \cite{L2}.  We will work with the specialization of his work to $SU(3)$ where the manifold invariant was first developed by Ohtsuki and Yamada. A general approach to constructing picture TQFTs is outlined in \cite{FNWW}.
\par
A second construction of this TQFT follows from the general procedure in \cite{Tur}.  Results of Blanchet have shown that the skein theory of Yokota allow for the construction of skein $SU(N)_k$ from a modular tensor category \cite{Blan}.

To describe the vector space $V(Y)$ given by the $2+1$ TQFTs, we fix a level $k$ of the theory.  A vector space $V(Y)$  is spanned by admissible labellings of a trivalent graph associated to $Y$.  Up to a homeomorphism into Euclidean space, we can think of $Y$ as bounding a standard handle body.  Then we can associate to $Y$, the spine of that handlebody, meaning a trivalent graph whose regular neighborhood is the handlebody.  A pants decomposition of $Y$ leads to which trivalent graph is used by taking the dual trivalent graph of that pants decomposition inside the handlebody $H$ bounded by $Y$.  An admissible labelling of a trivalent graph is defined as follows.  Each edge of the trivalent graph is associated with a fundamental projector, and each vertex a vector from the triangle space of the three incident edges.  Then $V(Y)$ is defined to be the free complex vector space spanned by each admissible labelling of this trivalent graph.  Equivalently $V(Y)$ is the reduced skein space, generated by framed links in $M$ and subject to relations coming from the kernel of a bilinear pairing.

Let $h:Y\rightarrow Y$ be an orientation preserving homeomorphism.  Now we look to define an action
\[V_h:V(Y)\rightarrow V(Y).\]
  This action arises through the skein theoretic description given by Ohtsuki and Yamada\cite{OY, Rob}.  We describe the action $V_h$ when $h$ is a positive Dehn twist about a curve $\gamma\subset Y$.  If we take $x\in V(Y)$ to be a labelling of the trivalent graph associated to $Y$ then we have $V_{T_\gamma}(x)$ is the labelling which is given by taking $x$ and adjoining the skein representing the label $x$ with the curve  $\gamma$, colored with $\Omega$, inserted inside the boundary with a $-1$ framing relative to the boundary,  where $\Omega$ is as defined in \cite{OY}.

The sliding property, balanced stabilization hold
\cite{OY}.

\begin{lem}
Let $Y$ be a closed, oriented surface, $h:Y\rightarrow Y$ an orientation preserving homeomorphism, and $V_h:V(Y)\rightarrow V(Y)$ the action of $h$ on $V_g$.  Let
\[C(a)=Z(Y\times I,a\times\{\frac{1}{2}\})\]
and $C(h(a))$ similar, where $a$ is given an orientation.  Either orientation could be used and this would define $\bar{C}(a)$, which could be used in all of the subsequent arguments.  Then
\[V_h C(a)V_h^{-1}=C(h(a)).\]
\end{lem}
\begin{proof}
To begin we note that it suffices to assume that $h$ is a positive Dehn twist about a curve $\gamma$.  Following our above description of the action we will have a skein stacked with $\gamma$ labeled with $\omega$ and framed placed over the curve $a$ placed over the curve $\gamma$ labeled with $\omega$ with the opposite framing, as seen on the left of Figure \ref{omega}.  Then using the handle slide property of $\omega$ we are able to pull the bottom copy of $\gamma$ through $a$ at the cost of a Dehn twist about $\gamma$.  This leaves us with $C(h(a))$ and two copies of $\gamma$ having opposite framings.  Finally we are able to apply the balanced stabilization property of $\omega$ to cancel these opposite copies.  This is visualized in Figure \ref{omega} in the case where $Y$ is the torus with each parallel square having edges identified, and $\gamma$ and $a$ are the meridian and longitude.

\begin{figure}[ht!b]
\centering
\includegraphics[scale=.3]{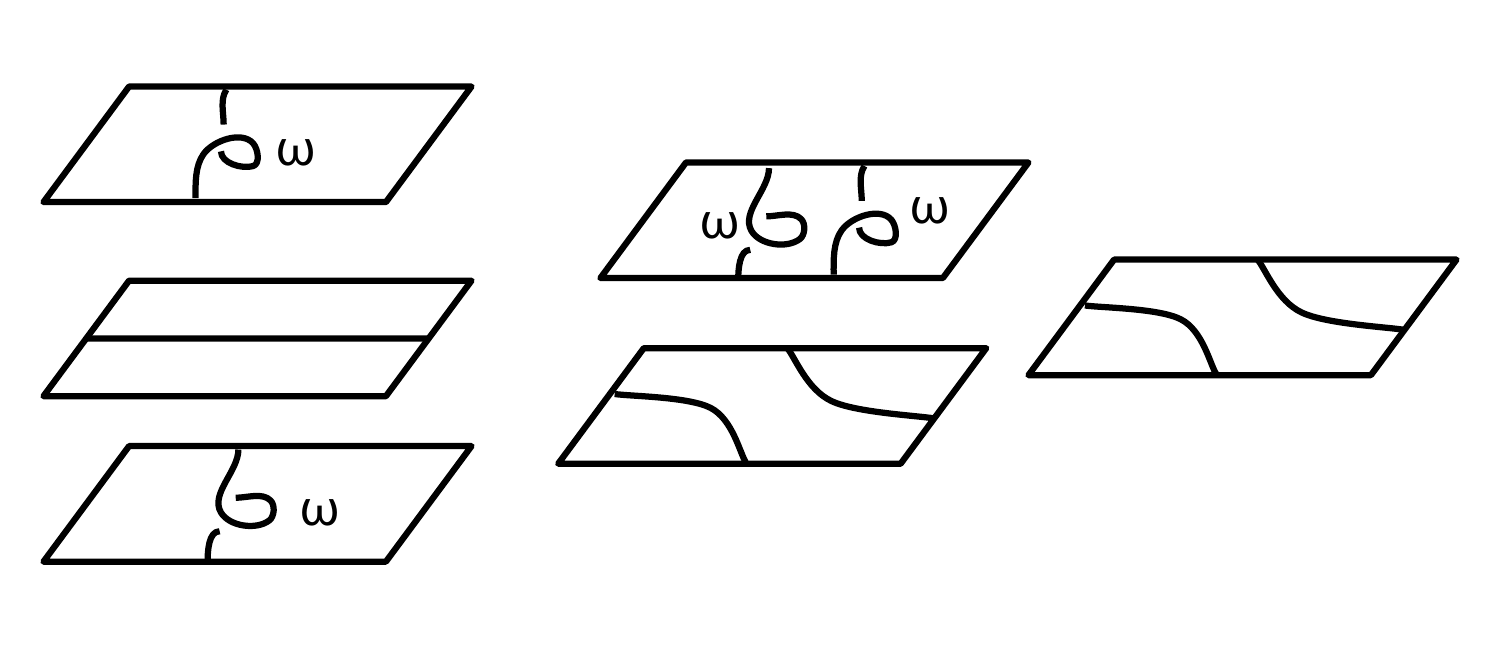}
\caption{•}\label{omega}
\end{figure}
\end{proof}
This result should be thought of as immediately coming from the handle slide invariance and the definitions all originating in the surgery descriptions of $3-$manifolds.

\section{Main Result: Skein Quantum $SU(3)$ asymptotic faithfulness}

\subsection{Lemmas}
\begin{lem}
Let $a$ and $b$ be two non-trivial, non-isotopic simple closed curves on a closed orientable surface $Y$.  Then there exists a pants decomposition of $Y$ such that $a$ is one of the decomposing curves and $b$ is a non-trivial \lq\lq graph geodesic" with respect to the decomposition in the sense that $b$ does not intersect any curve of the decomposition twice in a row.
\end{lem}

This is lemma $4.1$ of \cite{FWW}.

\begin{lem}
The identity tangle factors as the sum of projectors in the $A_2$ web.  In fact, if the level is high enough then the identity on $m$ positive strands and $n$ negative strands can be written as
\[P_{(m,n)}+\sum_{a+b<p+q} P_{(a,b)},\]
where appropriate \lq\lq connecting skeins" are included where needed.
\end{lem}
\begin{proof}
This fact follows from the recursive description of the projectors given above.  We introduce the notation $id(m,n)$ to be the identity tangle on $m$ upward oriented strands and $n$ downward oriented strands.  Then the claim becomes
\[id(m,n)=\sum x^{(m,n)}_{(a,b)}\circ P_{(a,b)}\circ y^{(a,b)}_{(m,n)}\]
Where the sum is taken over $(a,b)$ with $(a+b)\leq (m,n)$ and $x$ and $y$ are skeins that connect the appropriate number of strands to $id(m,n)$.  We proceed via induction, meaning assume that $i(a,b)$ can be factored for all $(a,b)$ with $a+b\leq m+n$.  Our first step is to note that $id(m,n)$ can be written as $id(m,n-1)$ with a downward oriented strand immediately to the right, think of $id(m,n-1)\otimes id(0,1)$.  With this in mind we can apply the factorization to the sub-tangle, see Figure $10$.
\begin{figure}[ht]
\centering
\includegraphics[scale=.35]{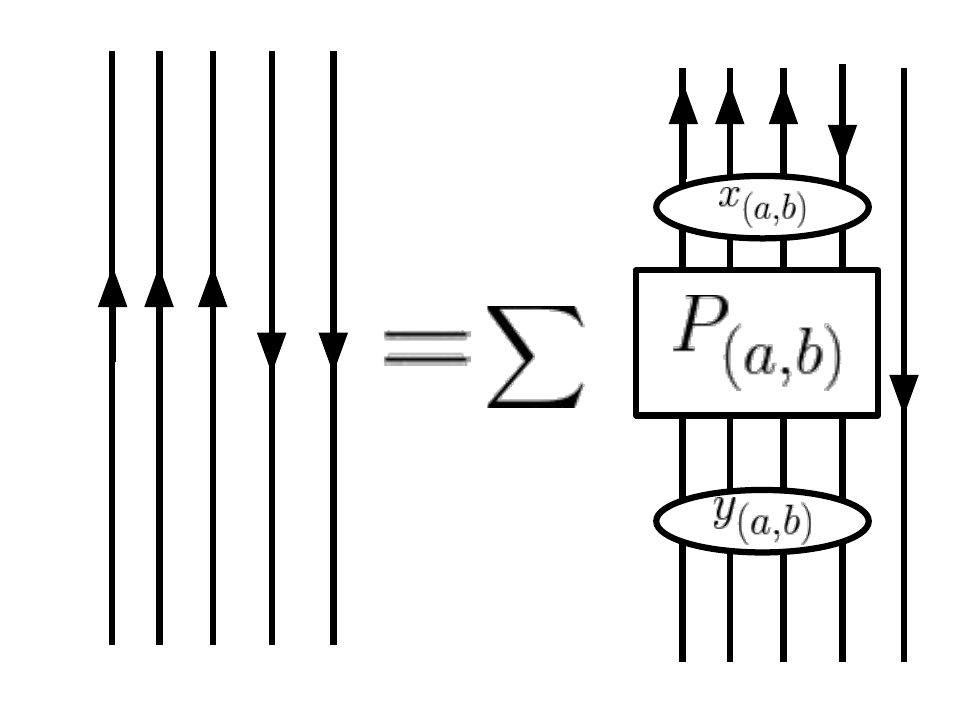}
\caption{}\label{turn}
\end{figure}

Now we utilize the recursive definition for the projectors to rewrite each term in our sum, as seen in figure $11$.
\begin{figure}
\centering
\includegraphics[scale=.55]{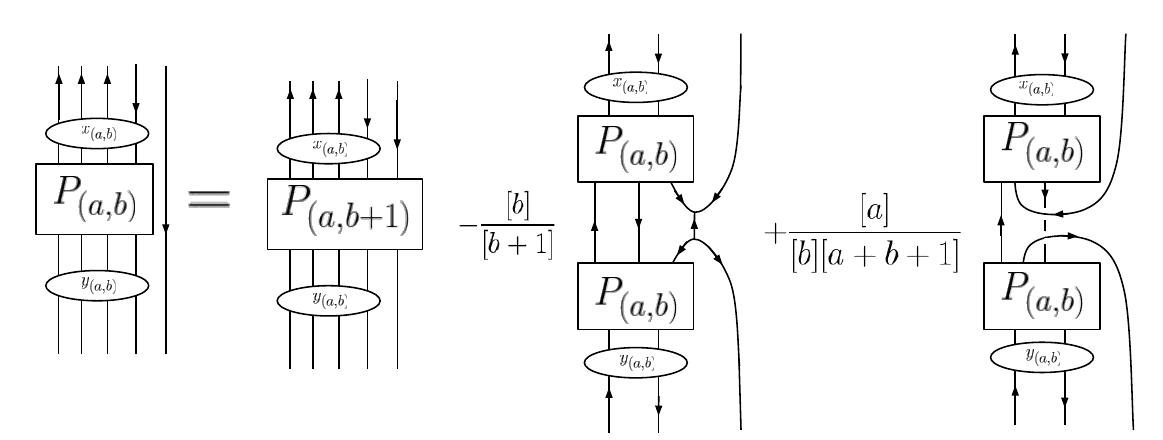}
\caption{}\label{turn}
\end{figure}
Then we see that this $3-$ term expansion gives us one term basically already in the form we desire, with the skein $x_{(a,b+1)}$ having a component that looks like $x_{(a,b)}$ with a disjoint strand immediately to the right.  The next two terms we see have some skeins above and below identity tangles on fewer than $a+b$ strands, as $(a+1)+(b-2)$ and $(a-1)+b$ both satisfy this.  This allows us to apply our inductive hypothesis on these "internal" identity tangles.  Then we can combine like terms, and in the process combining objects into a recursive construction for the $x$ and $y$ skeins.  We note that since $a+b<m+n-1$ we have $a+b+1<m+n$ for all $(a,b)$.  We also observe that the second half of our theorem holds as when applying this recursion to the lead term, we find the lead term in our new sum to be exactly $P_{(m,n)}$.
\end{proof}

\begin{prop}
Let $Y$ be a closed, oriented surface, $h$ an orientation preserving homeomorphism, and $V_h$ the skein quantum $SU(3)$ action.  Suppose there exists a simple closed curve $a\subset Y$ such that $h(a)$ is not isotopic (as a set) to $a$.  Then $V_h$ is a multiple of the identity for at most finitely many $k$.  That is, $h$ is eventually detected as $k$ increases .
\end{prop}

\begin{proof}
Since
\[V_h C(a)V_h^{-1}=C(h(a)),\]
 it suffices to show that for some $k$, $C(a)\neq C(h(a))$.
\par
By the graph geodesic lemma $2.1$ above, there exists a handlebody $H$ bounded by $Y$ such that $a$ bounds an embedded disk in $H$ and $h(a)$ is a non-trivial graph geodesic with respect to a spine $S$ of $H$.  This is illustrated below, where encircled region is exactly what is forbidden by the graph geodesic property.
\begin{figure}[ht]
\centering
\includegraphics[scale=.35]{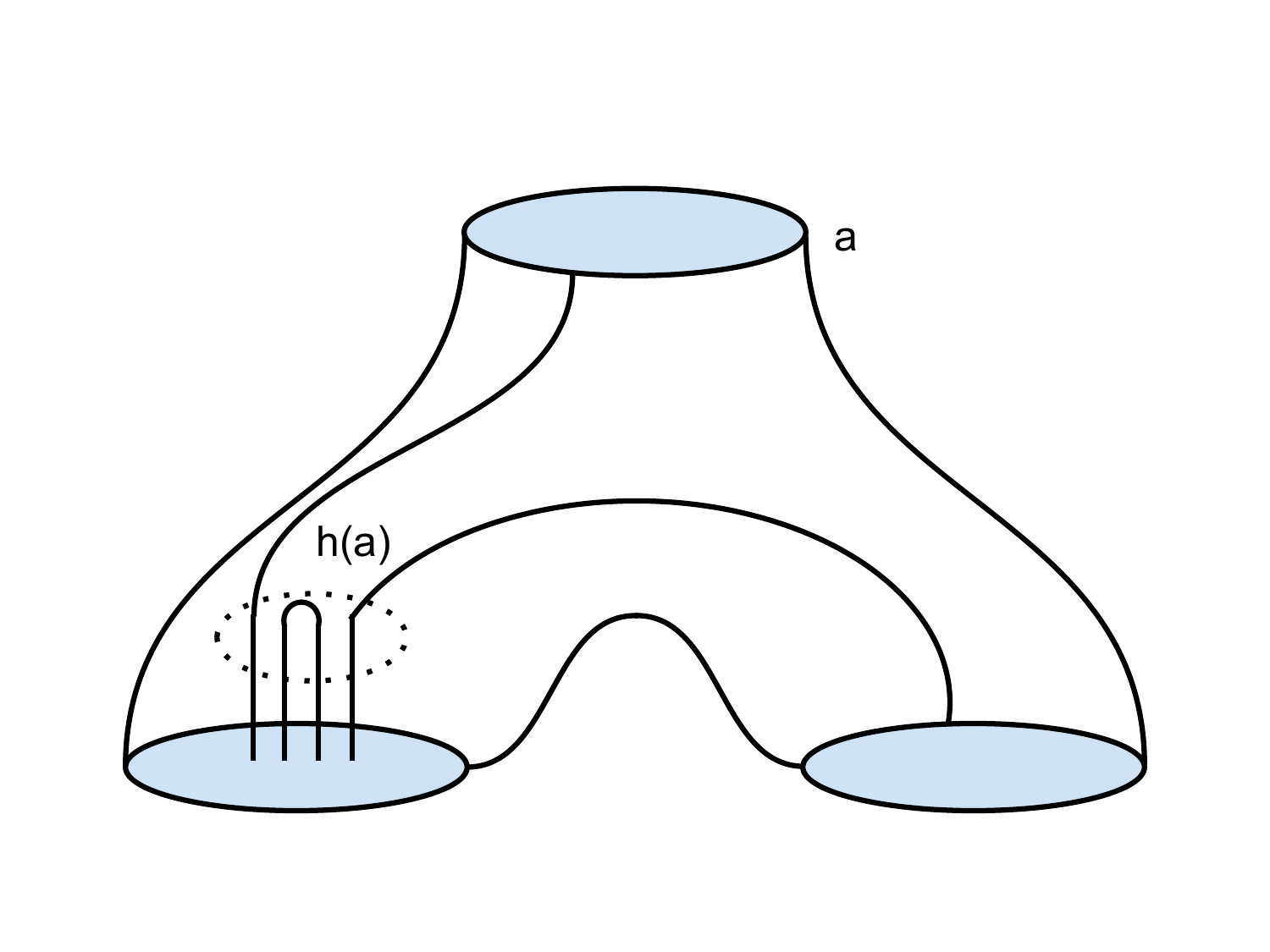}
\caption{}\label{turn}
\end{figure}
Now let $Z(H)\in V(Y)$ be the vector determined by $H$ with the empty labelling, and $Z(H,h(a))$ the vector determined by the pair $(H,h(a))$ (where $h(a)$ is pushed into the interior of $H$).  We have
\[C(a)(Z(H))=Z(H,a)=dZ(H)\]
as $a$ is taken to bound an embedded disk in $H$.  It is also true that
\[C(h(a))(Z(H))=Z(H,h(a)),\]
meaning it suffices to show that $Z(H,h(a))$ is not a multiple of $Z(H)$.
\par
Given a particular labelling $L=(l_1,...,l_n)$, where $l_i$ corresponds to a labelling of $v_i$ a vertex of the spine $S$, we construct a family of labellings $\{w_L\}$ of the spine $S$.   For each edge $e$ of $S$,  look at its dual disk $D$ and let $p$ be the number of times that $h(a)$ passes through $D$ where the orientation given by the orientation of boundary of $D$ and the orientation of $h(a)$ match, and $q$ the number of times $h(a)$ passes through $D$ with the opposite orientation.  Up to the potential of including a pillow in a collar near this disk we have can think of $h(a)$ as passing through the disk as the $(p,q)$ identity tangle, $id(p,q)$ as defined above.
\begin{figure}[ht]
\centering
\includegraphics[scale=.4]{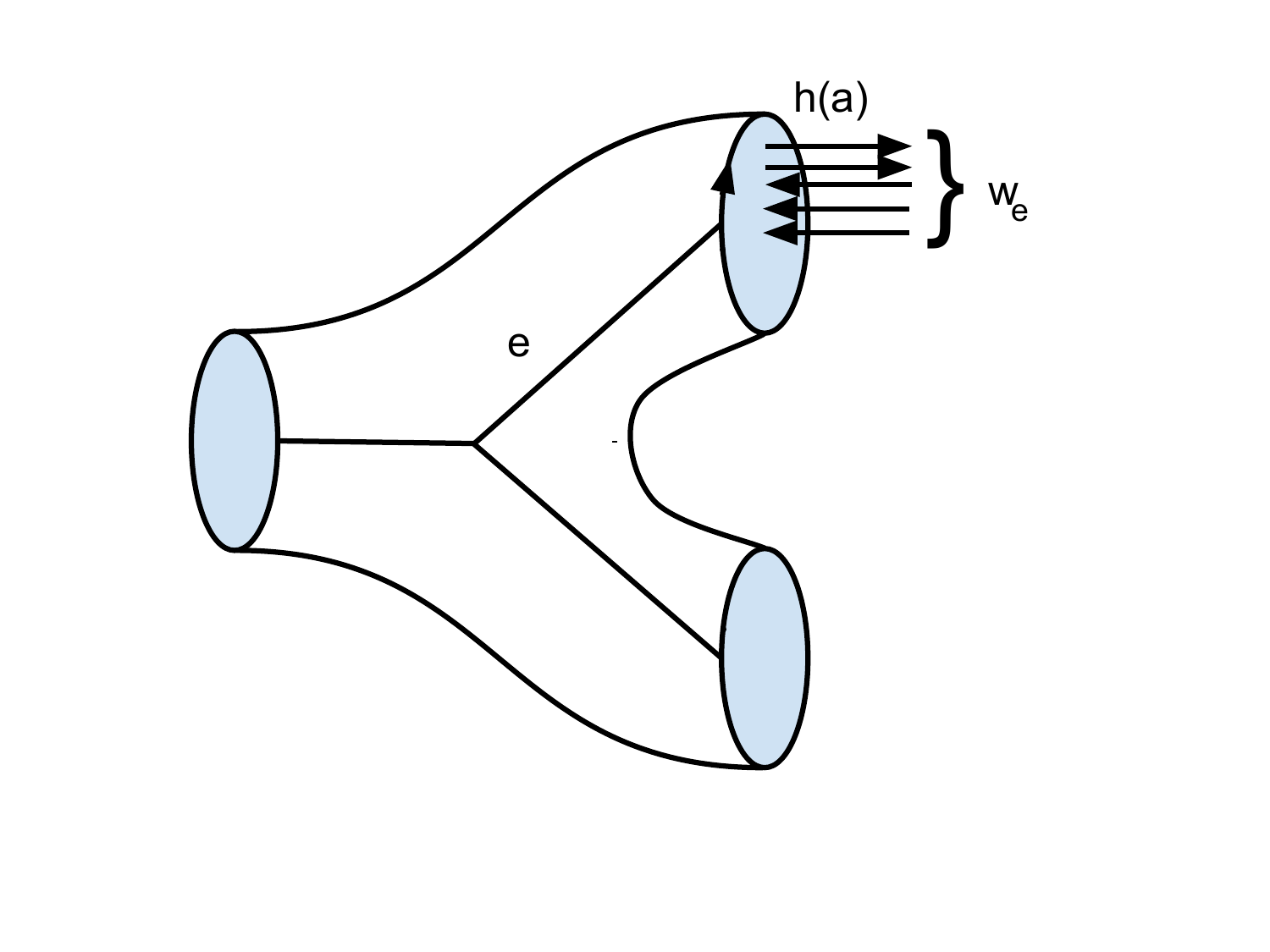}
\caption{•}\label{weight}
\end{figure}
Then the edge $e$ is labeled by $(p,q)$.  Finally define the complexity of a state vector as
\[m=\max(p_e+p_f+p_g, q_e+q_f+q_g\}\]
where $e,f,g$ are three edges meeting at a vertex, and the max is over all vertices.  Finally, pick a level $k\geq 2m$.
\par
Let $\{b_{w_L}\}$ be the basis vectors for $V(Y)$ determined by the labellings $w_L$ of $S$.  We claim that
\[Z(H,h(a))=\sum_{L}\lambda_L b_{w_L}+v,\]
where $\lambda_i\neq 0$ for all $i$, and $v$ consists of multiples of $b_v$ where $v$ is a label where
\[m_e+n_e< p_e+q_e\]
for each each $e$ of $S$.  Applying Lemma $4.2$ we see that each edge label for $Z(H,h(a))$ can be factored in the desired manner. We are left to look at the vertex of each trivalent graph.  This vertex is a skein in the triangle space of $(m_1,n_1), (m_2,n_2)$ and $(m_3,n_3)$.  Thus this skein is $\sum \lambda_i c_i$ where $c_i$ are the above basis vectors seen in Figure $6$.  Now we claim as least one of the $\lambda_i$ is nonzero.  As a skein, this trivalent vertex corresponds to the part of $h(a)$ which lives on a particular pair of paints, being pushed into the part of the handlebody bounded by this pair of pants.    Combinatorially, we can describe each component of $h(a)$ by which boundary components it connects.  The graph geodesic lemma tells us exactly that $h(a)$ never starts and ends at the same boundary component as seen in Figure \ref{turn}.  Then, introducing enough pillows, we are able to  rearrange each strand that connects boundary components with the same orientation next to each other. This gives us exactly a realization of a basis element as described for the intertwiner spaces described above, as seen in Figure $6$ for example.
 As we can do this for each vertex and we will have a nonzero scalar pulled from each one.  Therefore, $Z(H,h(a))$ cannot be a multiple of $Z(H)$ as it is not the empty labelling.
\end{proof}

\begin{thm}
Let $Y$ be a closed connected oriented surface and $\mathcal{M}(Y)$ its mapping class group. For every non-central $h\in\mathcal{M}$ there is an integer $r_0(h)$ such that for any $r\geq r_0(h)$ and any $A$ a primitive $6r$th root of unity, the operator
\[\langle h\rangle:V_A(Y)\rightarrow V_A(Y)\]
is not the identity,
\[\langle h\rangle\neq 1\in \mathcal{P}End(V_A),\]
the projective endomorphisms.  In particular, an appropriate infinite direct sum of quantum $SU(3)$ representations will faithfully represent these mapping class groups modulo center.
\end{thm}
\begin{proof}
If $h$ fixes all simple closed curves then $h$ must commute with all possible Dehn twists.  As Dehn twists generate the mapping class group of any surface we have that $h$ must be in the center.  Thus for any non-central $h\in\mathcal{M}$ that $h(a)$ is not isotopic to $a$, the main theorem follows from the above proposition.
\end{proof}

\end{document}